\documentclass[12pt]{amsart}
\usepackage{amscd}     
\usepackage{amssymb}
\usepackage{amsmath, amsthm, graphics, dsfont}
\usepackage{xypic}     
\LaTeXdiagrams        
\usepackage[all]{xy}
\xyoption{2cell} \UseAllTwocells \xyoption{frame} \CompileMatrices
\allowdisplaybreaks[3]
\usepackage{amsfonts}
\usepackage[normalem]{ulem}
\usepackage{hyperref}
\usepackage{tikz}
\usepackage{mathtools}
\usepackage{verbatim} 
\usepackage{MnSymbol}
\usepackage{latexsym}
\usepackage{epsfig}
\usepackage{enumerate}
\usepackage{times}
\usepackage{xcolor}
\usepackage{geometry}

\newgeometry{vmargin={25mm,25mm}, hmargin={25mm,25mm}}   

\newtheorem{prop}{Proposition}[section]

\newtheorem{numex}[prop]{Example}

\numberwithin{equation}{section}

\newcommand{\C}{\mathcal{C}}

\newcommand{\K}{\mathcal{K}}
\newcommand{\Z}{\mathcal{Z}}
\newcommand{\cO}{\mathcal{O}}
\newcommand{\G}{\mathcal{G}}

\date{\today}

\begin{document}

\title{A note on quantum K-theory of root constructions}

\author[Hsian-Hua Tseng]{Hsian-Hua Tseng}
\address{Department of Mathematics\\ Ohio State University\\ 100 Math Tower, 231 West 18th Ave. \\ Columbus,  OH 43210\\ USA}
\email{hhtseng@math.ohio-state.edu}

\begin{abstract}
We consider K-theoretic Gromov-Witten theory of root constructions. We calculate some genus $0$ K-theoretic Gromov-Witten invariants of a root gerbe. We also obtain a K-theoretic relative/orbifold correspondence in genus $0$.
\end{abstract}

\maketitle

\section{Introduction}

\subsection{\'Etale gerbes}
Let $X$ be a smooth projective variety over the complex numbers. An \'etale gerbe $\mathcal{G}$ over $X$, 
may be thought of as a fiber bundle over $X$ whose fibers are the classifying stack $BG$ of a certain finite group $G$. Geometric properties of $\mathcal{G}$ are of purely stack-theoretic nature. 

In \cite{hhps}, physical theories on an \'etale gerbe $\mathcal{G}$ are considered, leading to the formulation of {\em decomposition conjecture} (also known as {\em gerbe duality}). Interpreted in mathematics, the decomposition conjecture for $\mathcal{G}$ asserts that the geometry of $\mathcal{G}$ is equivalent to the geometry of a disconnected space $\widehat{\mathcal{G}}$ equipped with a $\mathbb{C}^*$-gerbe. The decomposition conjecture has been proven in several mathematical aspects in \cite{tt1}. 

\subsection{Gromov-Witten theory}
Gromov-Witten theory of a target $Z$ is defined using moduli stacks $\K_{g,n}(Z,d)$ of stable maps to $Z$. Gromov-Witten invariants of $Z$ are integrals of natural cohomology classes on $\K_{g,n}(Z,d)$ against the virtual fundamental class of $\K_{g,n}(Z,d)$.


The Gromov-Witten theory of an \'etale gerbe $\mathcal{G}$ has been studied, with a point of view towards the decomposition conjecture,  in various generalities in \cite{ajt1}, \cite{ajt3}, \cite{j}, \cite{tt2}, \cite{t0}. 

\subsection{Quantum K-theory}
Quantum K-theory, introduced in \cite{g}, \cite{l}, is the K-theoretic counterpart of Gromov-Witten theory. K-theoretic Gromov-Witten invariants of a target $Z$ are Euler characteristics of natural K-theory classes on $\K_{g,n}(Z,d)$ tensored with the {\em virtual structure sheaf} $\mathcal{O}_{\K_{g,n}(Z,d)}^{vir}$.
An extension of quantum K-theory to target Deligne-Mumford stacks is given in \cite{tt}.

Quantum Hirzebruch-Riemann-Roch theorems \cite{gt}, \cite{tt}, \cite{g1}, \cite{g2} imply that quantum K-theory can be determined by (cohomological) Gromov-Witten theory. Since (cohomological) Gromov-Witten theory of \'etale gerbes has been shown to satisfy the decomposition conjecture in many cases, it is natural to ask if quantum K-theory of an \'etale gerbe $\mathcal{G}$ can be studied with a viewpoint towards the decomposition conjecture. This note contains an attempt to address this for {\em root gerbes} over $X$ in genus $0$. 

\subsection{Root gerbes}
Given a line bundle $L\to X$ and an integer $r>0$, one can associate the stack $\sqrt[r]{L/X}$ of $r$-th roots of $L$, which is a smooth Deligne-Mumford stack whose points over an $X$-scheme $f: S\to X$ are
\begin{equation*}
\sqrt[r]{L/X}(S)=\{(M, \phi) \,|\, M\to S \text{ line bundle}, \phi: M^{\otimes r}\overset{\simeq}{\longrightarrow} f^*L \}.    
\end{equation*}
The coarse moduli space of $\sqrt[r]{L/X}$ is $X$. Furthermore, the natural map $\rho: \sqrt[r]{L/X}\to X$ has the structure of a $\mu_r$-gerbe. 

The strategy employed to study quantum K-theory of $\sqrt[r]{L/X}$ in this note is the same as that of \cite{ajt1}. Namely, we examine the structure of moduli stacks of genus $0$ stable maps to $\sqrt[r]{L/X}$ and apply pushforward results for virtual structure sheaves. The main result of this note is Proposition \ref{prop:main}.

\subsection{Root stacks}
Given a smooth irreducible divisor $D\subset X$ and an integer $r>0$, one can associate the stack $X_{D,r}$ of $r$-th roots of $X$ along $D$. In \cite{acw}, genus $0$ {\em relative} Gromov-Witten invariants of $(X,D)$ and Gromov-Witten invariants of $X_{D,r}$ are shown to be the same when $r$ is sufficiently large. Their proof uses pushforwards of virtual fundamental classes and an intermediate moduli space. In Section \ref{sec:root_stack}, we explain how to adapt their argument to obtain a similar result for genus $0$ K-theoretic Gromov-Witten invariants, see (\ref{eqn:KGW_rel_orb}).

\subsection{Outline}
The rest of this note is organized as follows. Section \ref{sec:QK_review} recalls notations used the definition of K-theoretic Gromov-Witten invariants of Deligne-Mumford stacks. In Section \ref{sec:structure_map} we discuss properties of the structure morphism for moduli stacks of genus $0$ stable maps to a root gerbe. In Section \ref{sec:pushforward} we discuss pushforwards of virtual structure sheaves. Section \ref{sec:inv} contains the proof of our main result and Section \ref{sec:banded_gerbe} discusses an extension of the main result to a more general class of gerbes. In Section \ref{sec:root_stack}, we discuss a K-theoretic version of relative/orbifold correspondence. 
In Section \ref{sec:discussion} we discuss some related questions.

\subsection{Acknowledgment}
This note is inspired by the results on virtual pushforwards in K-theory in \cite{chl1} and \cite{chl2}. It is a pleasure to thank the authors Y.-C. Chou, L. Herr, and Y.-P. Lee. The author also thanks E. Sharpe for discussions. The author is supported in part by Simons Foundation Collaboration Grant.

\section{Results}

\subsection{Quantum K-theory of target stacks}\label{sec:QK_review}
We begin with recalling the definition of K-theoretic Gromov-Witten invariants of Deligne-Mumford stacks, as given in \cite{tt}. Let $\Z$ be a smooth proper Deligne-Mumford stack with projective coarse moduli space $Z$. The moduli stack of $n$-pointed genus $g$ degree $d$ stable maps to $\Z$ is denoted by 
$\K_{g,n}(\Z,d)$.    
The detailed definition can be found in \cite{av}. It is known that $\K_{g,n}(\Z,d)$ is a proper Deligne-Mumford stack equipped with a perfect obstruction theory, see \cite{av}, \cite{agv}. Applying the recipe of \cite{l} to this perfect obstruction theory yields a virtual structure sheaf $\cO_{\K_{g,n}(\Z,d)}^{vir}$.    
There are evaluation maps 
$ev_i: \K_{g,n}(\Z, d)\to \bar{I}\Z$,    
where $\bar{I}\Z$ is the {\em rigidified inertia stack} of $\Z$. See \cite{agv} for more details on the construction of evaluation maps.

K-theoretic Gromov-Witten invariants of $\Z$ are Euler characteristics of the following form:
\begin{equation}\label{defn:K_inv}
\chi\left(\K_{g,n}(\Z,d), \cO_{\K_{g,n}(\Z,d)}^{vir}\otimes \bigotimes_{i=1}^n ev_i^* \alpha_i\right)\in \mathbb{Z}, \quad \alpha_1,...,\alpha_n\in K^*(\bar{I}\Z).    
\end{equation}

\subsection{Structure morphism}\label{sec:structure_map}
Sending a stable map $f: (\C, \{\Sigma_i\}_{i=1}^n)/S\to \Z$ to the induced map $\bar{f}: (C, \{\bar{\Sigma}_i\}_{i=1}^n)/S\to Z$ between coarse moduli spaces yields a morphism
\begin{equation}\label{eqn:structure_map}
\K_{g,n}(\Z,d)\to \K_{g,n}(Z,d).    
\end{equation}
We examine (\ref{eqn:structure_map}) in the special case $\Z=\sqrt[r]{L/X}$ and $g=0$.

As explained in \cite[Section 3.1]{ajt1}, the rigidified inertia stack of $\sqrt[r]{L/X}$ is a disjoint union of components $\bar{I}(\sqrt[r]{L/X})_g$ indexed by $g\in \mu_r$. As in \cite[Definition 3.3]{ajt1}, for $g_1,...,g_n\in \mu_r$, put
\begin{equation*}
\K_{0,n}(\sqrt[r]{L/X},d)^{\vec{g}}=\bigcap_{i=1}^{n}ev_i^{-1}(\bar{I}(\sqrt[r]{L/X})_{g_i}).    
\end{equation*}
In order for $\K_{0,n}(\sqrt[r]{L/X},d)^{\vec{g}}$ to be non-empty, the elements $g_1,...,g_n$ are required to satisfy certain condition, see \cite[Section 3.1]{ajt1}.

We consider the restriction of (\ref{eqn:structure_map}) to $\K_{0,n}(\sqrt[r]{L/X},d)^{\vec{g}}$:
\begin{equation}\label{eqn:structure_map1}
p: \K_{0,n}(\sqrt[r]{L/X},d)^{\vec{g}}\to \K_{0,n}(X,d).    
\end{equation}
The structure of the map $p$ has been analyzed in \cite{ajt1}. We reproduce \cite[Diagram (26)]{ajt1} as follows:
\begin{equation}\label{main_diagram}
\xymatrix{ 
\K_{0,n}(\sqrt[r]{L/X},d)^{\vec{g}}\ar[r]_{\quad \quad t}\ar[dr]_{s}\ar@/^1.5pc/[rrr]^{p} & P_n^{\vec{g}}\ar[d]\ar[r]_{r'}\ar@{}[dr] | {\square} & P\ar[d]_{q'}\ar[r]_{r}\ar@{}[dr] | {\square} & \K_{0,n}(X,d)\ar[d]^q  \\
\, &\mathfrak{Y}_{0,n,d}^{\vec{g}}\ar[r] & \mathfrak{M}_{0,n,d}^{tw}\ar[d]_{s'}\ar[r]\ar@{}[dr] | {\square} & \mathfrak{M}_{0,n,d}\ar[d]^{s''} \\
\, & \, & \mathfrak{M}_{0,n}^{tw}\ar[r] & \mathfrak{M}_{0,n}.
}
\end{equation}
Here $\mathfrak{M}_{0,n}$ is the stack of $n$-pointed genus $0$ prestable curves (see e.g. \cite{bs} for a discussion), and $\mathfrak{M}_{0,n}^{tw}$ is the stack of $n$-pointed genus $0$ prestable twisted curves (see \cite{o}). $\mathfrak{M}_{0,n,d}$ and $\mathfrak{M}_{0,n,d}^{tw}$ are variants of $\mathfrak{M}_{0,n}$ and $\mathfrak{M}_{0,n}^{tw}$ parametrizing prestable (twisted) curves weighted by $d\in H^2(X,\mathbb{Z})$, see \cite[Section 3.2]{ajt1} for an introduction and \cite{bs} and \cite{t} for further details.

In (\ref{main_diagram}), the stack $\mathfrak{Y}_{0,n,d}^{\vec{g}}$ is constructed in \cite[Definition 3.12]{ajt1} by applying the root construction to a certain divisor of $\mathfrak{M}_{0,n,d}$. It follows that the composition $\mathfrak{Y}_{0,n,d}^{\vec{g}}\to \mathfrak{M}_{0,n,d}^{tw}\to \mathfrak{M}_{0,n,d}$ is proper and birational.
The stacks $P$ and $P_n^{\vec{g}}$ are defined by cartesian squares. The map $s$ is defined by \cite[Lemma 3.18]{ajt1}.

\begin{numex}
When $X$ is a point, the line bundle $L$ is necessarily trivial. In this case $\sqrt[r]{L/X}=B\mu_r$. The moduli stacks $\K_{0,n}(B\mu_r)^{\vec{g}}$ and $\K_{0,n}(\text{pt})=\overline{\mathcal{M}}_{0,n}$ are smooth of expected dimensions. The morphism (\ref{eqn:structure_map1}) in this case has been studied in \cite{bc}. It is shown in \cite{bc} that there is a factorization 
$\K_{0,n}(B\mu_r)^{\vec{g}}\to \mathcal{N}\to \overline{\mathcal{M}}_{0,n}$,    
where $\K_{0,n}(B\mu_r)^{\vec{g}}\to \mathcal{N}$ is the stack of $r$-th roots of certain line bundle, and $\mathcal{N}\to \overline{\mathcal{M}}_{0,n}$ is a root construction.
\end{numex}

\subsection{Pushforward}\label{sec:pushforward}
We now examine obstruction theories. Since the map $s''$ is \'etale, the standard obstruction theory on $\K_{0,n}(X,d)$ relative to $\mathfrak{M}_{0,n}$ can be viewed as a obstruction theory $E_{\K_{0,n}(X,d)}^\bullet\to L_q^\bullet$ on $\K_{0,n}(X,d)$ relative to the morphism $q$. The stack $P$ can be equipped with an obstruction theory relative to the morphism $q'$ by pulling back $E_{\K_{0,n}(X,d)}^\bullet$. The stack $P_n^{\vec{g}}$ can be equipped with an obstruction theory relative to $\mathfrak{Y}_{0,n,d}^{\vec{g}}$ by pulling back the obstruction theory on $P$. 

Since both maps $s'$ and $\mathfrak{Y}_{0,n,d}^{\vec{g}}\to \mathfrak{M}_{0,n,d}^{tw}$ are \'etale \cite[Lemma 3.15]{ajt1}, the standard obstruction theory on $\K_{0,n}(\sqrt[r]{L/X},d)^{\vec{g}}$ relative to $\mathfrak{M}_{0,n}^{tw}$ can be viewed as an obstruction theory $E_{\K_{0,n}(\sqrt[r]{L/X},d)^{\vec{g}}}^\bullet\to L_s^\bullet$ on $\K_{0,n}(\sqrt[r]{L/X},d)^{\vec{g}}$ relative to the morphism $s$. 

By \cite[Lemma 4.1]{ajt1}, $E_{\K_{0,n}(X,d)}^\bullet$ pulls back to $E_{\K_{0,n}(\sqrt[r]{L/X},d)^{\vec{g}}}^\bullet$. We then have the following results on virtual structure sheaves. 
\begin{enumerate}
    \item Since $\mathfrak{Y}_{0,n,d}^{\vec{g}}\to \mathfrak{M}_{0,n,d}$ is proper and birational, by \cite[Theorem 1.12]{chl2}, we have 
    \begin{equation}\label{eqn:vir_pushforward1}
      (r\circ r')_*[\cO_{P_n^{\vec{g}}}^{vir}]=[\cO_{\K_{0,n}(X,d)}^{vir}].  
    \end{equation}

    \item By \cite[Theorem 3.19]{ajt1}, the map $t: \K_{0,n}(\sqrt[r]{L/X},d)^{\vec{g}}\to P_n^{\vec{g}}$ is a $\mu_r$-gerbe. Hence by \cite[Proposition 1.9]{chl2}, we have 
    \begin{equation}\label{eqn:vir_pushforward2}
      t_*[\cO_{\K_{0,n}(\sqrt[r]{L/X},d)^{\vec{g}}}^{vir}]=[\cO_{P_n^{\vec{g}}}^{vir}].  
    \end{equation}
    
\end{enumerate}
 
\subsection{Invariants}\label{sec:inv}
The evaluation maps on $\K_{0,n}(X,d)$ and $\K_{0,n}(\sqrt[r]{L/X},d)^{\vec{g}}$ fit into the following commutative diagram:
\begin{equation*}
\xymatrix{
\K_{0,n}(\sqrt[r]{L/X},d)^{\vec{g}}\ar[r]^{\quad  ev_i}\ar[d]^{p} & \bar{I}(\sqrt[r]{L/X})_{g_i}\ar[d]^{\bar{I}\rho} \\
\K_{0,n}(X,d)\ar[r]^{ev_i} & \bar{I}X=X.
}
\end{equation*}

Consider the descendant line bundles $L_1, ..., L_n\to \K_{0,n}(X,d)$ associated to the marked points. The following is the main result of this note:

\begin{prop}\label{prop:main}
For $\alpha_1,...,\alpha_n\in K^*(X)$ and $k_1,...,k_n\in \mathbb{Z}$, we have 
\begin{equation*}
\begin{split}
\chi\left(\K_{0,n}(\sqrt[r]{L/X},d)^{\vec{g}}, \cO_{\K_{0,n}(\sqrt[r]{L/X},d)^{\vec{g}}}^{vir}\otimes \bigotimes_{i=1}^n ((p^*L_i)^{\otimes k_i}\otimes ev_i^*((\bar{I}\rho)^*\alpha_i))  \right)\\
=\chi\left(\K_{0,n}(X,d), \cO_{\K_{0,n}(X,d)}^{vir}\otimes \bigotimes_{i=1}^n (L_i^{\otimes k_i}\otimes ev_i^*(\alpha_i))  \right).
\end{split}
\end{equation*}
\end{prop}
\begin{proof}
Since $\bar{I}\rho\circ ev_i=ev_i\circ p$, projection formula gives 
\begin{equation*}
p_*\left(\cO_{\K_{0,n}(\sqrt[r]{L/X},d)^{\vec{g}}}^{vir}\otimes \bigotimes_{i=1}^n ((p^*L_i)^{\otimes k_i}\otimes ev_i^*((\bar{I}\rho)^*\alpha_i))\right)=p_*(\cO_{\K_{0,n}(\sqrt[r]{L/X},d)^{\vec{g}}}^{vir})\otimes \bigotimes_{i=1}^n (L_i^{\otimes k_i}\otimes ev_i^*(\alpha_i)).     
\end{equation*}
Since $p=r\circ r'\circ t$, the result follows from (\ref{eqn:vir_pushforward1}) and (\ref{eqn:vir_pushforward2}).
\end{proof}

\subsection{Banded abelian gerbes}\label{sec:banded_gerbe}
Suppose $G$ is a finite abelian group. Suppose $\mathcal{G}\to X$ is a gerbe banded by $G$. Then the isomorphism class of $\mathcal{G}\to X$ is classified by the cohomology group $H^2(X, G)$, where $G$ is viewed as a constant sheaf on $X$. We say that $\mathcal{G}\to X$ is {\em essentially trivial} if the image of its class is trivial for maps $H^2(X,G)\to H^2(X,\mathbb{C}^*)$ induced by group homomorphisms $G\to \mathbb{C}^*$. Examples of essentially trivial gerbes include toric gerbes \cite{t0}.

Let $\mathcal{G}\to X$ be an essentially trivial gerbe over $X$. Then by \cite[Lemma A.2]{ajt1}, $\mathcal{G}$ is of the form
\begin{equation}
\mathcal{G}\simeq \sqrt[r_1]{L_1/X}\times_X\sqrt[r_1]{L_2/X}\times_X...\times_X\sqrt[r_k]{L_k/X}    
\end{equation}
where $L_1,...,L_k$ are line bundles over $X$ and $r_1,...,r_k$ are natural numbers.

Consider the morphism (\ref{eqn:structure_map}) in this case:
\begin{equation}\label{eqn:structure_map_G}
\K_{0,n}(\mathcal{G},d)\to \K_{0,n}(X,d).      
\end{equation}
By the alaysis of \cite[Appendix A]{ajt1}, (\ref{eqn:structure_map_G}) also fits into diagram like (\ref{main_diagram}), with a factorization 
\begin{equation}
\K_{0,n}(\mathcal{G},d)^{\vec{g}}\to P_n^{\vec{g}}\to \K_{0,n}(X,d).          
\end{equation}
Here $\vec{g}$ is defined in \cite[Definition A.5]{ajt1}.

The map $P_n^{\vec{g}}\to \K_{0,n}(X,d)$ is by construction virtually birational, hence we can apply \cite[Theorem 1.12]{chl2} to it. By \cite[Theorem A.6]{ajt1}, the map $\K_{0,n}(\mathcal{G},d)^{\vec{g}}\to P_n^{\vec{g}}$ is also a gerbe, so we can apply \cite[Proposition 1.9]{chl2} to it. Therefore, we may repeat the arguments in Section \ref{sec:inv} to extend Proposition \ref{prop:main} to essentially trivial banded abelian gerbes $\mathcal{G}\to X$.

\subsection{Root stacks}\label{sec:root_stack}
Let $D\subset X$ be a smooth irreducible divisor. For an integer $r>0$, one can construct the stack $X_{D,r}$ of $r$-th roots of $X$ along $D$, see \cite{c} and \cite[Appendix B]{agv}. The natural map 
\begin{equation}\label{eqn:map_XDr}
X_{D,r}\to X    
\end{equation}
 is an isomorphism over $X\setminus D$ and is a $\mu_r$-gerbe over $D$. Denote by $D_r\subset X_{D,r}$ the inverse image of $D$ under (\ref{eqn:map_XDr}).

It is shown in \cite{acw} that genus $0$ relative Gromov-Witten invariants of the pair $(X,D)$ are the same as Gromov-Witten invariants of $X_{D,r}$ for $r$ sufficiently large. Here we explain how their method can be adapted to K-theoretic Gromov-Witten theory.

By \cite{af}, there is an isomorphism between moduli spaces\footnote{We omit curve classes from notations.} of stable relative maps,
\begin{equation*}
\Psi: \overline{M}_{0,n}(X_{D,r}, D_r)\to \overline{M}_{0,n}(X,D),    
\end{equation*}
see also \cite[Theorem 2.1]{acw}. This implies an identification of virtual structure sheaves,
\begin{equation}\label{eqn:rel_rel_str_sheaf}
\Psi_*[\mathcal{O}_{\overline{M}_{0,n}(X_{D,r}, D_r)}^{vir}]=[\mathcal{O}_{\overline{M}_{0,n}(X,D)}^{vir}].    
\end{equation}
There is a natural map that forgets the relative structure
\begin{equation*}
\Phi:\overline{M}_{0,n}(X_{D,r}, D_r)\to \overline{M}_{0,n}(X_{D,r}).    
\end{equation*}
Assume that $r$ is sufficiently large. The proof of \cite[Theorem 2.2]{acw} implies that $\Phi$ is virtually birational. Hence by \cite[Theorem 1.12]{chl2}, we have
\begin{equation}\label{eqn:rel_orb_str_sheaf}
\Phi_*[\mathcal{O}_{\overline{M}_{0,n}(X_{D,r}, D_r)}^{vir}]=[\mathcal{O}_{\overline{M}_{0,n}(X_{D,r})}^{vir}].    
\end{equation}
Evaluation maps of these moduli spaces are compatible with $\Psi$ and $\Phi$, see \cite[Section 2.2]{acw}. It follows from (\ref{eqn:rel_rel_str_sheaf}) and (\ref{eqn:rel_orb_str_sheaf}) that for $\alpha_1,...,\alpha_k\in K^*(X)$, $\gamma_1,...,\gamma_l\in K^*(D)$, and $r$ sufficiently large, we have
\begin{equation}\label{eqn:KGW_rel_orb}
\begin{split}
&\chi\left(\overline{M}_{0,n}(X_{D,r}),\mathcal{O}_{\overline{M}_{0,n}(X_{D,r})}^{vir}\otimes \bigotimes_{i=1}^k L_i^{k_i}\otimes ev_i^*(\alpha_i)\otimes \bigotimes_{j=1}^l L_j^{m_j}\otimes ev_j^*(\beta_j) \right)\\
=&
\chi\left(\overline{M}_{0,n}(X,D),\mathcal{O}_{\overline{M}_{0,n}(X,D)}^{vir}\otimes \bigotimes_{i=1}^k L_i^{k_i}\otimes ev_i^*(\alpha_i)\otimes \bigotimes_{j=1}^l L_j^{m_j}\otimes ev_j^*(\beta_j) \right).
\end{split}
\end{equation}
We view (\ref{eqn:KGW_rel_orb}) as a correspondence between genus $0$ K-theoretic Gromov-Witten invariants of $(X,D)$ and $X_{D,r}$.

\section{Comments}\label{sec:discussion}

\subsection{On higher genus}
\subsubsection{Root gerbes}
For $h>0$, the genus-$h$ version of the morphism (\ref{eqn:structure_map1}),
\begin{equation}\label{eqn:structure_map2}
\K_{h,n}(\sqrt[r]{L/X},d)^{\vec{g}}\to \K_{h,n}(X,d),  \end{equation}
has been studied in \cite{ajt3}. The map (\ref{eqn:structure_map2}) is understood well enough so that a result on the pushforward of virtual fundamental {\em classes} is proven in \cite{ajt3}. However, pushforward of virtual structure {\em sheaves} under (\ref{eqn:structure_map2}) appears to be difficult. The key issue is that, in order to apply \cite[Proposition 1.9, Theorem 1.12]{chl2}, we need (\ref{eqn:structure_map2}) be factored into virtual birational maps and gerbes. A factorization of (\ref{eqn:structure_map2}) was obtained for more general banded gerbes in \cite[Diagram (41)]{ajt3}. In our setting this gives
\begin{equation}
\K_{h,n}(\sqrt[r]{L/X},d)^{\vec{g}}\to P_{h,n}^{\vec{g}}\to \K_{h,n}(X,d).  
\end{equation}
For a root gerbe $\sqrt[r]{L/X}\to X$, one can check that $P_{h,n}^{\vec{g}}\to \K_{h,n}(X,d)$ is also virtually birational. However, by the discussion of \cite[Section 6.2]{ajt3}, the map $\K_{h,n}(\sqrt[r]{L/X},d)^{\vec{g}}\to P_{h,n}^{\vec{g}}$ is a composition of two maps, one has degree $1/r$ and the other had degree $r^{2h}>1$. The degree $r^{2h}$-map cannot possibly be a gerbe. Hence \cite[Proposition 1.9]{chl2} is not applicable to $\K_{h,n}(\sqrt[r]{L/X},d)^{\vec{g}}\to P_{h,n}^{\vec{g}}$. This prevents us from obtaining genus-$h$ version of Proposition \ref{prop:main}. 

\subsubsection{Root stacks}
The relative/orbifold correspondence in cohomological Gromov-Witten theory has been extended to higher genus in \cite{ty}. A K-theoretic relative/orbifold correspondence in higher genus is an interesting question. It is unlikely that virtual pushforwards used in genus $0$ will be enough in higher genus. Some foundational work in K-theoretic Gromov-Witten theory is required in order to follow the arguments in \cite{ty}.

\subsection{On virtual pushforward}
There are many situations in cohomological Gromov-Witten theory in which ``virtually birational'' maps occur, see \cite{HW} for a detailed list. In addition, we note that the morphism $u$ in \cite[Lemma 4.16]{cclt} is virtually birational. Hence we can apply \cite[Theorem 1.12]{chl2} to obtain a calculation of the K-theoretic $J$-function of weighted projective spaces. Since such a result is a special case of the work \cite{z} on quantum K-theory of toric stacks, we do not persue it in details.

\subsection{On decomposition conjecture}
Consider an \'etale gerbe $\G\to X$. As defined in (\ref{defn:K_inv}), K-theoretic Gromov-Witten invariants of $\G$ have insertions coming from the K-theory $K^*(\bar{I}\G)$ of the rigidified inertia stack $\bar{I}\G$. Since $\G\subset \bar{I}\G$ is a connected component, the K-theory $K^*(\G)$ of $\G$ is a direct summand of $K^*(\bar{I}\G)$. The proof of Proposition \ref{prop:main} only allows classes in $K^*(\bar{I}\sqrt[r]{L/X})$ pulled back from $X$. Studying K-theoretic Gromov-Witten invariants of $\sqrt[r]{L/X}$ with other kinds of insertions requires new ideas. 

For root gerbes $\G\to X$ arising in toric geometry, e.g. weighted projective spaces and more general toric gerbes, it may be possible to study the decomposition conjecture by analyzing the K-theoretic $I$-functions calculated in \cite{z} in a manner similar to \cite{t0}. An additive decomposition of the K-theory $K(\bar{I}\G)$ is a basic question.


\begin{thebibliography}{12}

\bibitem{acw} D. Abramovich, C. Cadman, J. Wise, {\em Relative and orbifold Gromov-Witten invariants}, Algebr. Geom. 4 (2017), no. 4, 472--500.

\bibitem{af} D. Abramovich, B. Fantechi, {\em Orbifold techniques in degeneration formulas}, Ann. Sc. Norm. Super. Pisa Cl. Sci. (5) 16 (2016), no. 2, 519--579.

\bibitem{agv} D. Abramovich, T. Graber, A. Vistoli, {\em Gromov-Witten theory of Deligne-Mumford stacks}, Amer. J. of Math. 130 (2008), no. 5, 1337--1398.

\bibitem{av} D. Abramovich, A. Vistoli, {\em Compactifying the space of stable maps}, J. Amer. Math. Soc. 15 (2002), 27--75.

\bibitem{ajt1} E. Andreini, Y. Jiang, H.-H. Tseng, {\em Gromov-Witten theory of root gerbes I: structure of genus 0 moduli spaces},  J. Differential Geom. Vol. 99, no. 1 (2015), 1--45.

\bibitem{ajt3} E. Andreini, Y. Jiang, H.-H. Tseng, {\em Gromov-Witten theory of banded gerbes over schemes}, arXiv:1101.5996.

\bibitem{bs} Y. Bae, J. Schmidt, {\em Chow rings of stacks of prestable curves. I}, Forum Math. Sigma 10, Paper No. e28, 47 p. (2022).

\bibitem{bc} A. Bayer, C. Cadman, {\em Quantum cohomology of $[\mathbb{C}^N/\mu_r]$}, Compos. Math. 146, No. 5, 1291--1322 (2010).

\bibitem{c} C. Cadman, {\em Using stacks to impose tangency conditions on curves}, Amer. J. Math. 129 (2007), no. 2, 405--427.

\bibitem{chl1} Y.-C. Chou, L. Herr, Y.-P. Lee, {\em The log product formula in quantum K--theory}, Math. Proc. Cambridge Philos. Soc. 175 (2023), no. 2, 225--252.

\bibitem{chl2} Y.-C. Chou, L. Herr, Y.-P. Lee, {\em Higher Genus Quantum K--theory}, arXiv:2305.10137.

\bibitem{cclt} T. Coates, A. Corti, Y.-P. Lee, H.-H. Tseng, {\em The quantum orbifold cohomology of weighted projective spaces}, Acta Math. 202, No. 2, 139--193 (2009).

\bibitem{g} A. Givental, {\em On the WDVV equation in quantum K-theory}, Mich. Math. J. 48, Spec. Vol., 295--304 (2000).

\bibitem{g1} A. Givental, {\em Permutation-equivariant quantum K-theory IX. Quantum Hirzebruch-Riemann-Roch in all genera}, arXiv:1709.03180. 


\bibitem{g2} A. Givental, {\em Permutation-equivariant quantum K-theory X. Quantum Hirzebruch-Riemann-Roch in genus 0}, arXiv:1710.02376.  

\bibitem{gt} A. Givental, V. Tonita, {\em The Hirzebruch-Riemann-Roch theorem in true genus-0 quantum K-theory}, in: ``Symplectic, Poisson, and noncommutative geometry'', 43--91, Mathematical Sciences Research Institute Publications 62, Cambridge University Press (2014).

\bibitem{hhps} S. Hellerman, A. Henriques, T. Pantev, E. Sharpe, {\em Cluster decomposition, T-duality, and gerby CFTs}, Adv. Theor. Math. Phys., 11 (5) (2007), 751--818.

\bibitem{HW} L. Herr, J. Wise, {\em Costello's pushforward formula: errata and generalization}, Manuscripta Math. 171 (2023), no. 3--4, 621--642, arXiv:2103.10348. 


\bibitem{j} P. Johnson, {\em Equivariant GW theory of stacky curves}, Commun. Math. Phys. 327, No. 2, 333--386 (2014).

\bibitem{l} Y.-P. Lee, {\em Quantum K--theory. I: Foundations}, Duke Math. J. 121, No. 3, 389--424 (2004).

\bibitem{o} M. Olsson, {\em On (Log) twisted curves}, Compos. Math. 143, No. 2, 476--494 (2007).

\bibitem{tt1} X. Tang, H. -H. Tseng, {\em Duality theorems for \'etale gerbes on orbifolds}, Adv. Math. 250 (2014), 496--569.

\bibitem{tt2} X. Tang, H. -H. Tseng, {\em A quantum Leray-Hirsch theorem for banded gerbes}, J. Differential Geom. 119 (3), 459--511, (2021). 

\bibitem{tt} V. Tonita, H.-H. Tseng, {\em Quantum orbifold Hirzebruch-Riemann-Roch theorem in genus zero}, arXiv:1307.0262.

\bibitem{t0} H.-H. Tseng, {\em On Gromov-Witten theory of toric gerbes}, Albanian J. Math. 14 (2020), no. 1, 3--23.

\bibitem{t} H.-H. Tseng, {\em On the tautological rings of stacks of twisted curves}, to appear in Kyoto Journal of Mathematics, arXiv:2205.05529.

\bibitem{ty} H.-H. Tseng, F. You, {\em Higher genus relative and orbifold Gromov-Witten invariants}, Geom. Topol. 24 (2020), no. 6, 2749--2779.

\bibitem{z} M. Zhang, {\em Quantum K-theory of toric stacks}, preprint available on the author's website.

\end{thebibliography}
\end{document}